\documentclass{rmmcart}
\usepackage{amssymb,amscd}
\usepackage{pstricks}
\usepackage{graphicx}
\usepackage{amsthm}

\numberwithin{equation}{section}

\newtheorem{theorem}{Theorem}[section]

\newtheorem{coro}[theorem]{Corollary}
\newtheorem{lemma}[theorem]{Lemma}
\newtheorem{definition}[theorem]{Definition}

\newtheorem{conjecture}[theorem]{Conjecture}

\theoremstyle{remark}

\theoremstyle{definition}
\newtheorem{example}[theorem]{Example}
\newtheorem{remark}[theorem]{Remark}

\title{Generalized Palindromic Continued Fractions}
\date{\today}   

\author{David M. Freeman}
\address{University of Cincinnati Blue Ash College\\ 9555 Plainfield Road, Cincinnati, Ohio 45236}
\email{david.freeman@uc.edu}

\begin{document} 


\begin{abstract}
In this paper we introduce a generalization of palindromic continued fractions as studied by Adamczewski and Bugeaud. We refer to these generalized palindromes as $m$-palindromes, where $m$ ranges over the positive integers. We provide a simple transcendency criterion for $m$-palindromes, extending and slightly refining an analogous result of Adamczewski and Bugeaud. We also provide methods for constructing examples of $m$-palindromes. Such examples allow us to illustrate our transcendency criterion and to explore the relationship between $m$-palindromes and stammering continued fractions, another concept introduced by Adamczewski and Bugeaud.
\end{abstract}

\maketitle

\section{Introduction}\label{S:intro}

In \cite{AB07} (see also \cite{AA07} and \cite{AB10}), Adamczewski and Bugeaud study \textit{palindromic} continued fractions. That is, continued fractions whose sequences of partial quotients exhibit a certain form of mirror symmetry (see Section \ref{S:prelims} for a more precise definition). Amongst other applications, the authors used their results about palindromic continued fractions to provide new examples of transcendental numbers with bounded partial quotients, to provide new proofs of the transcendency of certain well-known continued fractions (such as the Thue-Morse continued fraction; see \cite{AB07_AMM}), and to answer questions related to the Littlewood Conjecture (see \cite{AB06_JLMS}). The main purpose of the present paper is to generalize and slightly refine the following result from \cite{AB07} (see also \,\cite[Theorem 35]{AA07}). 

\begin{theorem}[Adamczewski, Bugeaud; \cite{AB07}]\label{T:large_pals} 
Let $\alpha=[a_0,a_1,\dots]$ denote an irrational continued fraction. If $\alpha$ is palindromic, then $\alpha$ is either transcendental or quadratic.
\end{theorem} 

We generalize the above result by introducing \textit{$m$-palindromic} continued fractions, where $m$ ranges over the positive integers. Instead of exhibiting the mirror symmetry of palindromes ($1$-palindromes in our terminology), an $m$-palindromic continued fraction exhibits a sort of multiplicative symmetry (which we will precisely define in Section \ref{S:prelims}). Our main result is the following theorem.

\begin{theorem}\label{T:main}
Let $\alpha=[a_0,a_1,\dots]$ denote an irrational continued fraction. If there exists $m\in\mathbb{N}$ such that $\alpha$ is $m$-palindromic, then $\alpha$ is either 
\begin{enumerate}
  \item{transcendental, or}
  \item{a quadratic irrational such that $\alpha/m$ is reduced.}
\end{enumerate}
\end{theorem}

Via results from \cite{Burger05} we prove the following corollary to Theorem \ref{T:main} in the case that $m=1$. 

\begin{coro}\label{C:m_one}
Let $\alpha=[a_0,a_1,\dots]$ denote an algebraic irrational continued fraction. If $\alpha$ is $1$-palindromic, then $\alpha$ is a reduced quadratic irrational that is equivalent to its algebraic conjugate. 
\end{coro}

In Example \ref{X:non_equivalent}, we construct a quadratic irrational $2$-palindromic continued fraction $\alpha$ such that neither $\alpha$ nor $\alpha/2$ is equivalent to its algebraic conjugate, thus demonstrating that Corollary \ref{C:m_one} does not extend in a straightforward way to $m$-palindromes for $m\geq2$.

While constructing examples to illustrate the above results, we  discuss the relationship between $m$-palindromes and \textit{stammering} continued fractions as defined and studied in the work of Adamczewski and Bugeaud (see \cite[Section 2]{AB05}, for example). In other words, we discuss the extent to which the partial quotients of an $m$-palindrome exhibit large repetitive patterns. Our examples reveal stronger differences between $m$-palindromes ($m\geq 2$) and stammering continued fractions than exist between $1$-palindromes and stammering continued fractions, particularly in the case of positive palindromic density (cf.\,\cite[Section 7]{AB07}).

Section \ref{S:prelims} provides the basic definitions and notation that we will use. Section \ref{S:lemmas} records several auxiliary results with which we will construct examples and prove Theorem \ref{T:main}. In Section \ref{S:zeros} we pause to define and discuss extended continued fractions. In Section \ref{S:main_result}, we prove Theorem \ref{T:main}. In Section \ref{S:stammering} we discuss the relationship between $m$-palindromes and stammering continued fractions. Section \ref{S:maximal} contains a conjecture regarding a generalization of palindromic density. Finally, in Section \ref{S:equivalent}, we prove Corollary \ref{C:m_one}.

\section{Basic Definitions and Notation}\label{S:prelims}

Write $\mathbb{N}$ to denote the set of positive integers $\{1,2,3,\dots\}$, and $\mathbb{N}_0$ to denote the non-negative integers $\{0,1,2,3,\dots\}$. Given $i\in\mathbb{N}_0$, a \textit{(rational) continued fraction} is denoted by 
$[a_0,a_1,\dots,a_i]$, where $a_0\in\mathbb{N}_0$ and, for each $k$ such that $1\leq k\leq i$, we have $a_k\in\mathbb N$. For each $k$ such that $0\leq k\leq i$, the numbers $a_k$ are called \textit{partial quotients}. This definition easily extends to irrational continued fractions of the form $[a_0,a_1,\dots]$ corresponding to infinite sequences of partial quotients.  

It is convenient to use the language of sequences to define continued fractions, where the terms of the sequences correspond to partial quotients. Given any set $\mathcal{A}$, we write $\mathcal{A}^*$ to denote the collection of finite sequences composed of elements of $\mathcal{A}$ (here we include the empty sequence $\varepsilon$). For example, we write $\mathbb{N}^*$ and $\mathbb{N}_0^*$ to denote the collections of finite sequences composed of positive and non-negative integers, respectively. Thus we write $\mathbb{N}_0\times\mathbb{N}^*$ to denote the collection of finite sequences of the form $(a_0,a_1,\dots,a_i)$ such that $a_0\in\mathbb{N}_0$ and, for each $1\leq k\leq i$, we have $a_k\in\mathbb{N}$. For $A\in\mathcal{A}^*$, we write $|A|$ to denote the number of terms in $A$, and say that $A$ has \textit{length} $|A|$. For two sequences $A$ and $B$ in $\mathcal{A}^*$, we write $AB$ to denote the concatenation of $A$ and $B$. In this case we say that $A$ is a \textit{prefix} for the sequence $AB$. Given $A\in\mathcal{A}^*$ and $k\in \mathbb{N}$, we write $A^k$ to denote the concatenation of $k$ copies of $A$. We write $\overline{A}$ to denote the infinite sequence $AAA\dots$, and describe such a sequence as \textit{periodic}, with \textit{period} $A$. If there exist $B,C\in\mathcal{A}^*$ such that $A=B\overline{C}$, then we say that $A$ is \textit{eventually periodic}.

The notation $\mathbb{N}^{\mathbb{N}_0}$ denotes infinite sequences of the form $(a_0,a_1,a_2,\dots)$ in which every term is a positive integer. For example, given $A\in \mathbb{N}^*$, we have $\overline{A}\in\mathbb{N}^{\mathbb{N}_0}$. We write $\mathbb{N}_0\times\mathbb{N}^\mathbb{N}$ to denote infinite sequences of the form $(a_0,a_1,a_2,\dots)$, where $a_0\in \mathbb{N}_0$ and, for each $k\in \mathbb{N}$, we have $a_k\in\mathbb{N}$. 

Given a sequence $A\in\mathbb N^*$ and any positive real number $x$, we denote by $A^x$ the sequence $A^{\lfloor x \rfloor}A'$, where $A'$ is a prefix of $A$ of length $\left\lceil (x- \lfloor x\rfloor)|A|\right\rceil$. Here $\lceil x\rceil$ denotes the smallest integer greater than or equal to $x$, and $\lfloor x\rfloor$ denotes the largest integer less than or equal to $x$.

We remind the reader of the bijective correspondence between infinite sequences $A\in\mathbb{N}_0\times\mathbb{N}^\mathbb{N}$ and their  corresponding continued fractions $[A]$. However, given two sequences $A$ and $B$ in $\mathbb{N}_0\times\mathbb{N}^*$, it may be the case that $A\not= B$ while $[A]=[B]$. For example, $[1,1,1]=[1,2]$. To avoid ambiguities, we will take special care to distinguish between a given sequence of partial quotients $A\in\mathbb{N}_0\times\mathbb{N}^*$ and the corresponding continued fraction $[A]$. 

Given $A=(a_0,\dots,a_i)\in\mathbb{N}_0\times\mathbb{N}^*$, for each $k$ such that $0\leq k \leq i$, we refer to the numbers $[a_0,\dots,a_k]$ as \textit{convergents}. We write $[a_0,\dots,a_k]=p_k/q_k$, such that $p_k$ and $q_k$ have no common factors. Thus each sequence $(a_0,\dots,a_i)$ yields a sequence of convergents $(p_0/q_0, p_1/q_1, \dots, p_i/q_i)$.

Given $A=(a_0,\dots,a_i)\in\mathbb{N}^*$, the \textit{reversal} of $A$ is defined as $\widetilde A = (a_i,\dots,a_0)$. If there exists a number $m\in\mathbb{N}$ such that $[A]=m[\widetilde{A}]$, we say that $A$ is an \textit{$m$-palindrome}, or \textit{$m$-palindromic}. Note that when $m=1$, we recover the definition of a palindromic sequence as used in \cite{AB07}. In the case of an irrational continued fraction $\alpha=[A]=[a_0,a_1,\dots]$, we say that $\alpha$ (equivalently, $A$) is an $m$-palindrome, or $m$-palindromic, if there exist infinitely many indices $i$ for which the prefix  $(a_0,\dots,a_i)$ is an $m$-palindrome.

\section{Auxiliary Results}\label{S:lemmas}

Given $i,j\in\mathbb{N}$, our notational convention in this section shall be to write $A=(a_0,\dots,a_i)\in\mathbb{N}_0\times\mathbb{N}^*$ and $B=(b_0,\dots,b_j)\in\mathbb{N}_0\times\mathbb{N}^*$. Furthermore, we shall write $(p_0/q_0,\dots,p_i/q_i)$ and $(r_0/s_0,\dots,r_j/s_j)$ to denote the sequences of convergents corresponding to $A$ and $B$, respectively. The convergents corresponding to $\widetilde{A}$ and $\widetilde{B}$ shall be denoted by $(\tilde p_0/\tilde q_0,\dots,\tilde p_i/\tilde q_i)$ and $(\tilde r_0/\tilde s_0,\dots,\tilde r_j/\tilde s_j)$, respectively. 

The majority of this section (and the paper as a whole) will rely on the following matrix identity.

\begin{lemma}\label{L:matrix}
Given $A=(a_0,\dots,a_i)\in\mathbb{N}_0\times\mathbb{N}^*$, we have
\begin{equation}\label{E:matrix_id}
\left(\begin{array}{ll}
			p_i & p_{i-1} \\
			q_i & q_{i-1} \end{array}\right)=%
				\left(\begin{array}{ll}
						a_0 & 1 \\
						1   & 0 \end{array}\right)%
				\left(\begin{array}{ll}
						a_1 & 1 \\
						1   & 0 \end{array}\right)%
				\cdots
				\left(\begin{array}{ll}
						a_i & 1 \\
						1   & 0 \end{array}\right).
						\end{equation}
\end{lemma}

\begin{proof}
See, for example, \cite[Lemma 2.8]{BPSZ14}.
\end{proof}

Lemma \ref{L:matrix} leads to the following two results of which we will make frequent use.

\begin{lemma}\label{L:basic}
Given $i,m\in\mathbb{N}$ and $A=(a_0,\dots,a_i)\in\mathbb{N}^*$, we have 
\begin{equation}\label{E:transpose}
\left(\begin{array}{ll}
			p_i     & q_i \\
			p_{i-1} & q_{i-1} \end{array}\right)=%
		\left(\begin{array}{ll}
			\tilde{p}_i & \tilde{p}_{i-1} \\
			\tilde{q}_i & \tilde{q}_{i-1} \end{array}\right).
\end{equation}
Furthermore, $A$ is $m$-palindromic if and only if
\begin{equation}\label{E:mpal}
\left(\begin{array}{ll}
			p_i		& mq_i \\
			p_{i-1}	& q_{i-1}\\ \end{array}\right)=%
		\left(\begin{array}{ll}
			\tilde p_i	& \tilde q_i \\
			m\tilde p_{i-1} & \tilde q_{i-1}\end{array}\right).
\end{equation}
Equivalently, $A$ is $m$-palindromic if and only if $mq_i=p_{i-1}$.
\end{lemma}

\begin{proof}
Given $[A]=p_i/q_i$, via Lemma \ref{L:matrix} one can write 
\[\left(\begin{array}{ll}
			p_i & p_{i-1} \\
			q_i & q_{i-1} \end{array}\right)=%
				\left(\begin{array}{ll}
						a_0 & 1 \\
						1   & 0 \end{array}\right)%
				\left(\begin{array}{ll}
						a_1 & 1 \\
						1   & 0 \end{array}\right)%
				\cdots
				\left(\begin{array}{ll}
						a_i & 1 \\
						1   & 0 \end{array}\right).
\]
This equality then yields 
\[\left(\begin{array}{ll}
			p_i     & q_i \\
			p_{i-1} & q_{i-1} \end{array}\right)=%
		\left(\begin{array}{ll}
			p_i & p_{i-1} \\
			q_i & q_{i-1} \end{array}\right)^t=%
		\left(\begin{array}{ll}
			\tilde{p}_i & \tilde{p}_{i-1} \\
			\tilde{q}_i & \tilde{q}_{i-1} \end{array}\right).\]
Here we write $M^t$ for the transpose of a matrix $M$. Thus we verify (\ref{E:transpose}).

Assume $m[\widetilde A]=[A]$, so that $m\tilde p_i/\tilde q_i=p_i/q_i$. Since $p_i=\tilde{p}_i$, it follows that $\tilde{q}_i=mq_i$. Finally, we have $p_{i-1}=\tilde q_i=mq_i=m\tilde p_{i-1}$. Conversely, if (\ref{E:mpal}) is true, then $p_i/q_i=m\tilde p_i/\tilde q_i$, and $m[\widetilde A]=[A]$. We also note that if $mq_i=p_{i-1}$, then $m\tilde p_i/\tilde q_i=mp_i/p_{i-1}=p_i/q_i$.
\end{proof}

\begin{lemma}\label{L:split}
Given $A=(a_0,\dots,a_i)\in\mathbb{N}_0\times\mathbb{N}^*$ and $B=(b_0,\dots,b_j)\in\mathbb{N}^*$, we have
\[[AB]=\frac{p_i[B]+p_{i-1}}{q_i[B]+q_{i-1}}=\frac{p_ir_j+p_{i-1}s_j}{q_ir_j+q_{i-1}s_j}.\]
\end{lemma}

\begin{proof}
This follows from Lemma \ref{L:matrix} (see \cite[Lemma 2.11]{BPSZ14}).
\end{proof}

We now prove a series of lemmas that will provide methods for constructing $m$-palindromes.

\begin{lemma}\label{L:construct}
For $i,j,m\in\mathbb{N}$, suppose $A=(a_0,\dots,a_i)\in\mathbb{N}^*$ and $B=(b_0,\dots,b_j)\in\mathbb{N}^*$ satisfy $m[\widetilde A]=[B]$. The sequence $BA$ is $m$-palindromic if and only if $m[a_i,\dots,a_1]=[b_0,\dots,b_{j-1}]$.
\end{lemma}

\begin{proof}[Proof of Lemma \ref{L:construct}]
By Lemma \ref{L:split}, we may write
\begin{equation*}\label{E:quotients}
[\widetilde A\widetilde B]=\frac{\tilde p_i\tilde{r}_j+\tilde p_{i-1}\tilde{s}_j}{\tilde q_i\tilde{r}_j+\tilde q_{i-1}\tilde{s}_j} \quad \text{ and } \quad [BA]=\frac{r_j{p}_i+r_{j-1}{q}_i}{s_j{p}_i+s_{j-1}{q}_i}.
\end{equation*}
By (\ref{E:transpose}), we have $\tilde p_i\tilde{r}_j+\tilde p_{i-1}\tilde{s}_j=p_ir_j+q_ir_{j-1}$. From our assumption that $m\tilde p_i/\tilde q_i=r_j/s_j$, we have $\tilde q_i\tilde r_j=ms_jp_i$. Via (\ref{E:transpose}), we have $[a_i,\dots,a_1]=\tilde p_{i-1}/\tilde q_{i-1}=q_i/q_{i-1}$. Thus $m[a_i,\dots,a_1]=[b_0,\dots,b_{j-1}]$ if and only if $mq_i/q_{i-1}=r_{j-1}/s_{j-1}$. The equality $mq_i/q_{i-1}=r_{j-1}/s_{j-1}$ is in turn equivalent to $mq_is_{j-1}=r_{j-1}q_{i-1}=\tilde s_j \tilde q_{i-1}$. It follows that, under the stated assumptions, $m[\widetilde A \widetilde B]=[BA]$ if and only if $m[a_i,\dots,a_1]=[b_0,\dots,b_{j-1}]$.
\end{proof}

\begin{example} Given $m,n\in\mathbb{N}$, write $B=(m^2,1)\in\mathbb N^*$ and $A=(m)\in\mathbb N^*$. We note that $m[A^{2n}]=[B^n]$ and $m[A^{2n-1}]=[B^{n-1},m^2]$. By Lemma \ref{L:construct}, $B^nA^{2n}$ is an $m$-palindrome. 
\end{example}

The next lemma provides a means of constructing $m$-palindromic irrational continued fractions with periodic sequences of partial quotients.

\begin{lemma}\label{L:periodic_m_pals}
For $m\in\mathbb N$ and $A\in\mathbb N^*$, if $A$ is an $m$-palindrome, then $A^2$ is an $m$-palindrome.
\end{lemma}

\begin{proof}
Analogous to the proof of Lemma \ref{L:construct}, we need only show that $\tilde p_i \tilde p_i + \tilde p_{i-1} \tilde q_i=p_ip_i+p_{i-1}q_i$ and that $\tilde q_i\tilde p_i+\tilde q_{i-1}\tilde q_i=m(q_ip_i+q_{i-1}q_i)$. These equalities follow from (\ref{E:transpose}) and (\ref{E:mpal}), respectively.
\end{proof}

\begin{remark}
Suppose $A=(a_0,\dots,a_i)\in\mathbb{N}^*$ is an $m$-palindrome. Lemma \ref{L:periodic_m_pals} and Lemma \ref{L:construct} imply that $m[a_i,\dots,a_1]=[a_0,\dots,a_{i-1}]$.
\end{remark}

\begin{example}\label{X:short_pal}
For $m,n\in\mathbb N$, we note that $(mn,n)$ is $m$-palindromic. It follows from Lemma \ref{L:periodic_m_pals} that the irrational number $[\overline{mn,n}]$ is an $m$-palindrome. 
\end{example}

The final lemma of this section will enable the construction of irrational $m$-palindromic continued fractions with non-periodic sequences of partial quotients.

\begin{lemma}\label{L:pert_symm_middle}
For $m\in\mathbb N$ and $A,B\in\mathbb N^*$, if $A$ and $B$ are $m$-palindromes, then $ABA$ is an $m$-palindrome. 
\end{lemma}

\begin{proof}
By two applications of Lemma \ref{L:split}, we have
\[[\widetilde{A}\widetilde{B}\widetilde{A}]=\frac{\tilde p_i\tilde r_j\tilde p_i+\tilde p_i\tilde r_{j-1}\tilde q_i+\tilde p_{i-1}\tilde s_j\tilde p_i+\tilde p_{i-1}\tilde s_{j-1}\tilde q_i}{\tilde q_i\tilde r_j\tilde p_i+\tilde q_i\tilde r_{j-1}\tilde q_i+\tilde q_{i-1}\tilde s_j\tilde p_i+\tilde q_{i-1}\tilde s_{j-1}\tilde q_i}.\]
The lemma then follows from repeated applications of Lemma \ref{L:basic}, as in the proofs of Lemma \ref{L:construct} and Lemma \ref{L:periodic_m_pals}.
\end{proof}

\section{Extended Continued Fractions}\label{S:zeros}

In this section we use sequences $B\in\mathbb{N}_0$ to define \textit{extended} continued fractions. This definition will enable the construction of Examples \ref{X:non_purely_stammering} and \ref{X:non_equivalent}.

Let $B=(b_0,\dots,b_j)\in\mathbb{N}_0^*$. In analogy to Lemma \ref{L:matrix}, we write
\begin{equation}\label{E:extended_def}
\left(\begin{array}{ll}
			b_0	& 1 \\
			1	& 0 \end{array}\right)
\dots
\left(\begin{array}{ll}
			b_j	& 1 \\
			1	& 0 \end{array}\right)=
\left(\begin{array}{ll}
			r_j	& r_{j-1} \\
			s_j	& s_{j-1} \end{array}\right)
\end{equation}
and define $[B]=r_j/s_j$ (provided $s_j\not=0$). We again use the term \textit{convergent} to refer to the quotient $r_j/s_j$. The sequences of \textit{continuants} $(r_0,\dots,r_j)$ and $(s_0,\dots,s_j)$  satisfy the recurrence relations
\begin{equation}\label{E:recurrence}
r_k=b_kr_{k-1}+r_{k-2}\qquad s_k=b_ks_{k-1}+s_{k-2}\qquad(k=1,2,\dots,j).
\end{equation}
Here we define $r_{-1}=1$ and $s_{-1}=0$.

Given $B\in\mathbb{N}_0^*$ such that $[B]$ is defined, we now describe how to obtain $B'\in\mathbb{N}_0\times\mathbb{N}^*$ such that $[B]=[B']$. We refer to such $B'$ as a \textit{simplification} of $B$. To this end, let $B=(b_0,\dots,b_j)\in\mathbb{N}_0^*$ be such that $[B]$ is defined. In the case that $b_j=0$, we observe that (\ref{E:recurrence}) implies $[B]=[b_0,\dots,b_{j-2}]$. Here we note that if $j=1$ then $[B]$ is undefined, and thus we may assume that $j\geq2$. This observation may be applied at most finitely many times in order to obtain a sequence $B''\in\mathbb{N}_0^*$ such that $[B]=[B'']$, and such that either the final term of $B''$ is positive or $B''=(0)$. 

Next, we note that, for two integers $x,y\in\mathbb{N}_0$, we have
\begin{equation}\label{E:addition}
\left(\begin{array}{ll}
			x	& 1	\\
			1	& 0 \end{array}\right)
\left(\begin{array}{ll}
			0	& 1	\\
			1	& 0 \end{array}\right)
\left(\begin{array}{ll}
			y	& 1	\\
			1	& 0 \end{array}\right)=
\left(\begin{array}{ll}
			x+y	& 1	\\
			1	& 0 \end{array}\right).
\end{equation}
Thus, given $B''\in\mathbb{N}_0^*$ as above, we may apply (\ref{E:addition}) at most finitely many times to obtain $B'\in\mathbb{N}_0\times\mathbb{N}^*$ such that $[B]=[B'']=[B']$. Thus we obtain a desired simplification of $B$.

\begin{remark}\label{R:extended}
Because extended continued fractions are defined via (\ref{E:extended_def}), it is straightforward to verify that Lemmas \ref{L:basic}, \ref{L:split},  \ref{L:periodic_m_pals}, and \ref{L:pert_symm_middle} are valid for extended continued fractions provided that all extended continued fractions under consideration are well-defined and positive. 
\end{remark}

\section{Criterion for transcendency}\label{S:main_result}   

The proof of Theorem \ref{T:main} closely follows that of Theorem \ref{T:large_pals}. In particular, we rely on the following result of Schmidt (see \cite[Theorem 4.1]{AB07}).

\begin{theorem}[Schmidt; \cite{Schmidt67}]\label{T:Schmidt}
Let $\alpha$ be an irrational real number that is not quadratic. If there exists a real number $w>3/2$ and infinitely many triples of integers $(p,q,r)$ with $q>0$ such that 
\[\max\left\{\left|\alpha-\frac{p}{q}\right|,\left|\alpha^2-\frac{r}{q}\right|\right\}<\frac{1}{q^w},\]
then $\alpha$ is transcendental.
\end{theorem}

\begin{proof}[Proof of Theorem \ref{T:main}]
Suppose $\alpha=[A]$ for $A=(a_0,a_1,\dots)\in\mathbb N^{\mathbb{N}_0}$. Let $i$ denote an index for which $m[a_i,\dots,a_0]=[a_0,\dots,a_i]$. By Lemma \ref{L:basic}, we note that $p_{i-1}=mq_i$. We then observe that 
\[\left|\alpha^2-\frac{mp_i}{q_{i-1}}\right|=\left|\alpha^2-\frac{p_i}{q_{i}}\frac{p_{i-1}}{q_{i-1}}\right|\leq\left|\alpha-\frac{p_i}{q_i}\right|\left|\alpha+\frac{p_{i-1}}{q_{i-1}}\right|+\frac{\alpha}{q_{i}q_{i-1}},\]
where the final inequality follows from the fact that 
\begin{align*}
\left(\alpha-\frac{p_i}{q_i}\right)\left(\alpha+\frac{p_{i-1}}{q_{i-1}}\right)&=\alpha^2+\alpha\left(\frac{p_{i-1}}{q_{i-1}}-\frac{p_i}{q_i}\right)-\frac{p_i}{q_i}\frac{p_{i-1}}{q_{i-1}}\\
&=\alpha^2+\frac{\alpha(-1)^i}{q_{i}q_{i-1}}-\frac{p_i}{q_i}\frac{p_{i-1}}{q_{i-1}}
\end{align*}
(see \cite[Corollary 2.15]{BPSZ14}). Since, $|\alpha-p_i/q_i|\leq 1$, we have
\[\left|\alpha^2-\frac{mp_i}{q_{i-1}}\right|<(1+2\alpha)\left|\alpha-\frac{p_i}{q_i}\right|+\frac{\alpha}{q_iq_{i-1}}.\]
By \cite[Theorem 2.25]{BPSZ14}, we know that $|\alpha-p_i/q_i|<1/(q_iq_{i-1})$, which yields
\begin{equation}\label{E:schmidt}
\left|\alpha^2-\frac{mp_i}{q_{i-1}}\right|<\frac{1+3\alpha}{q_iq_{i-1}}<\frac{1+3\alpha}{q_{i-1}^2}.
\end{equation}

Let $w$ denote any real number strictly between $3/2$ and $2$. There exists an index $i_0$ (determined by $\alpha$ and the choice of $w$) such that any index $i$ satisfying (\ref{E:schmidt}) such that $i\geq i_0$ will also satisfy
\begin{equation}\label{E:goal}
\left|\alpha^2-\frac{mp_i}{q_{i-1}}\right|<\frac{1}{q_{i-1}^w}.
\end{equation}
By assumption, there are infinitely many indices $i\geq i_0$ satisfying (\ref{E:schmidt}), and thus (\ref{E:goal}). Since we have $|\alpha-p_{i-1}/q_{i-1}|<1/q_{i-1}^2<1/q_{i-1}^w$ (see again \cite[Theorem 2.25]{BPSZ14}), we appeal to Theorem \ref{T:Schmidt} in light of the triples $(p_{i-1},q_{i-1},mp_i)$ to conclude that $\alpha$ is transcendental or quadratic.

For the remainder of the proof we assume that $\alpha=[A]$ is a quadratic irrational. By Lagrange's Theorem, $A$ is eventually periodic. In other words, there exist sequences $U,W\in\mathbb N^*$ such that $A=U\overline W$. Here we can assume that $U$ does not end in a copy of $W$. To prove that $\alpha/m$ is reduced, we proceed as follows. By assumption, there exists a subsequence $i_k\to+\infty$ such that, for each $k\in\mathbb N$, we have $m[a_{i_k},\dots,a_0]=[a_0,\dots,a_{i_k}]$. Since $[a_0,\dots,a_{i_k}]\to\alpha$ as $k\to+\infty$, we also have $[a_{i_k},\dots,a_0]\to\alpha/m$ as $k\to+\infty$. By taking a subsequence of the subsequence $(i_k)_{k=1}^{+\infty}$ (if necessary) we can assume that, for every $k\in\mathbb N$, we have $[a_0,\dots,a_{i_k}]=[UW^{j_k}W']$, where $j_k$ is a non-decreasing sequence of positive integers tending to $+\infty$, and $W'$ is a prefix of $W$. Here we allow for the possibility that $W'=W$. Since $j_k\to+\infty$, taking a limit yields $\widetilde A=X\overline Y$, where we write $X=\widetilde W'$ and $Y=\widetilde W$.  

Suppose $W=(w_1,w_2,\dots,w_n)$ and $W'=(w_1,\dots,w_r)$ for some $r\leq n$. If $r=n$, then $\widetilde A$ is purely periodic, with period $Y=\widetilde W$. If $r<n$, then $\widetilde A$ is purely periodic, with period $Z=(w_r,\dots,w_1,w_n,\dots,w_{r+1})$. In either case, Galois' theorem allows us to conclude that $[\widetilde A]=[A]/m$ is reduced.
\end{proof}

To construct transcendental $m$-palindromes, we utilize Lemma \ref{L:pert_symm_middle} along with \textit{generalized perturbed symmetries} (see \cite{Mendes81}, \cite{AS98}). Given $A\in\mathbb N^*$, we define the generalized perturbed symmetry $S_A:\mathbb N^*\to\mathbb N^*$ to be $S_A(X)=XAX$. Given a non-decreasing sequence of positive integers $n_k$, we also define $S_{A,k}:\mathcal{\mathbb N}^*\to\mathcal{\mathbb N}^*$ to be 
\begin{equation}\label{E:perturbed_S}
S_{A,k}(X)=XA^{n_k}X.
\end{equation}
and $T_{A,k}(X)=S_{A,k}\circ S_{A,k-1}\circ\dots\circ S_{A,1}(X)\in\mathbb N^*$. Set $\tau_{A,k}(X)=[T_{A,k}(X)]$, and write $\tau_A(X)=[T_A(X)]$ to denote the irrational continued fraction such that, for every $k\in\mathbb N$, the sequence $T_{A,k}(X)$ is a prefix of $T_A(X)$. Note that $T_A(X)$ is well defined because, for any $Y\in\mathbb N^*$, the sequence $S_{A,k}(Y)$ begins in $Y$. By Lemmas \ref{L:periodic_m_pals} and \ref{L:pert_symm_middle}, for each $k\in\mathbb N$, the sequence $T_{A,k}(X)$ is an $m$-palindrome if $A,X\in\mathbb N^*$ are $m$-palindromes. Therefore, the same is true of $T_A(X)$.

We note that $\tau_A(X)$ is quite similar to the continued fractions produced by \textit{generalized perturbed symmetry systems} as defined in \cite[Section 7]{ABD06}. Yet, since the collection $\{S_{A,k}\}_{k=1}^{+\infty}$ need not be finite, $\tau_A(X)$ need not fall within the immediate purview of \cite[Theorem 7.1]{ABD06}.

We also note that continued fractions of the form $\tau_A(X)$ are \textit{Maillet-Baker} continued fractions, as described in \cite[Section 3]{AB10}. Therefore, by satisfying certain conditions on the sequence $\{n_k\}_{k=1}^{+\infty}$ used in the construction of $\tau_A(B)$, continued fractions of the form $\tau_A(B)$ can be analyzed via \cite[Theorem 3.1]{AB10}. However, these conditions are not always satisfied, as in Example \ref{X:ST_number}.

\begin{example}\label{X:ST_number}
We now construct an irrational $2$-palindrome of the form $\tau_A(B)$ that does not satisfy the assumptions of \cite[Theorem 3.1]{AB10} or of \cite[Theorem 7.1]{ABD06}. In particular, we will see that $n_k/|T_{A,k-1}(B)|\to0$ as $k\to+\infty$ (thus failing the assumptions of \cite[Theorem 3.1]{AB10}), and the construction of $\tau_A(B)$ will require an infinite collection of generalized perturbed symmetries (thus failing the assumptions of \cite[Theorem 7.1]{ABD06}). 

To this end, let $A=(2,1)$ and $B=(2,1,1,3,1)$. Note that both $A$ and $B$ are $2$-palindromic. Define the sequence $\{n_k\}_{k=1}^{+\infty}$ as in (\ref{E:perturbed_S}) to be $n_k=k$. Then we have 
\begin{align*}
\tau_A(B)&=[T_A(B)]=[BABAABABAAABABAABAB\dots]\\
&=[2,1,1,3,1,2,1,2,1,1,3,1,2,1,2,1,2,1,1,3,1,2,1,\dots].
\end{align*}
By the remarks above, $\tau_A(B)$ is a $2$-palindrome. Furthermore, it is straightforward to verify that $|T_{A,k-1}(B)|>2^{k-1}-1$, and so we have $n_k/|T_{A,k-1}(B)|<k/(2^{k-1}-1)\to0$ as $k\to+\infty$. We also note that the collection $\{S_{A,k}\}_{k=1}^{+\infty}$ consists of infinitely many distinct generalized perturbed symmetries. To verify that $\tau_A(B)$ is transcendental, we observe that, given any $k\in\mathbb N$, the sequence $T_A(B)$ contains a copy of $A^k$. It follows that $T_A(B)$ is not eventually periodic, and thus $\tau_A(B)$ is not quadratic. Theorem \ref{T:main} then implies that $\tau_A(B)$ is transcendental.
\end{example}

\section{Stammering Continued Fractions}\label{S:stammering}

In this section we investigate the relationship between $m$-palindromes and so-called \textit{stammering} continued fractions, as defined in the work of Adamczewski and Bugeaud. We focus on the following three  versions of stammering continued fractions as described in \cite{AB05}, \cite{ABD06}, and \cite{AB10}. 

\begin{definition}
For any real number $w>1$, we say that an irrational continued fraction $\alpha=[A]$ satisfies Condition $(*)_w$ provided that $\alpha$ is not quadratic and there exist sequences $V_k\in\mathbb N^*$ such that
\begin{enumerate}
  \item{for $k\in\mathbb N$, the sequence $V_k^w$ is a prefix of $A$, and}
  \item{the sequence $\{|V_k|\}_{k=1}^{+\infty}$ is strictly increasing.}
\end{enumerate}
\end{definition}

\begin{definition}
For any real number $w>1$, we say that an irrational continued fraction $\alpha=[A]$ satisfies Condition $(*)_{\widehat w}$ provided that $\alpha$ is not quadratic and there exist sequences $V_k\in\mathbb N^*$ such that 
\begin{enumerate}
 \item{for $k\in\mathbb N$, the sequence $V_k^w$ is a prefix of $A$,}
 \item{the sequence $\{|V_{k+1}|/|V_k|\}_{k=1}^{+\infty}$ is bounded, and}
  \item{the sequence $\{|V_k|\}_{k=1}^{+\infty}$ is strictly increasing.}
\end{enumerate}
\end{definition} 

\begin{definition}\label{D:double_star}
For any real numbers $w>1$ and $w'\geq0$, we say that an irrational continued fraction $\alpha=[A]$ satisfies Condition $(**)_{w,w'}$ provided that $\alpha$ is not quadratic and there exist sequences $U_k,V_k\in\mathbb N^*$ such that 
\begin{enumerate}
  \item{for $k\in\mathbb N$, the sequence $U_kV_k^w$ is a prefix of $A$,}
  \item{the sequence $\{|U_k|/|V_k|\}_{k=1}^{+\infty}$ is bounded above by $w'$, and}
  \item{the sequence $\{|V_k|\}_{k=1}^{+\infty}$ is strictly increasing.}
\end{enumerate}
\end{definition}

We note that the continued fraction $\tau_A(B)$ from Example \ref{X:ST_number} satisfies Condition $(*)_{\widehat w}$ for  $w=2$ (cf.\,\cite[p.\,887]{AB10}). Therefore, it can be analyzed via \cite[Theorem 1]{AB05} (see also \cite[Theorem 2.1]{AB10}). However, \cite[Theorem 1]{AB05} relies on the Schmidt Subspace Theorem (see \cite{Schmidt72}), and so Theorem \ref{T:main} provides a theoretically simpler transcendency criterion for $\tau_A(B)$.

In \cite[Section 7]{AB07}, Adamczewski and Bugeaud point out that any non-quadratic irrational $1$-palindromic continued fraction that has positive palindromic density satisfies Condition $(*)_w$ for some $w>1$. In order to confirm that the analogous statement is \textit{not} true for $m$-palindromes when $m\geq2$, we define \textit{$m$-palindromic density} as follows. Given any irrational $m$-palindromic continued fraction $\alpha=[A]$, let $\{P_k(A)\}_{k=0}^{+\infty}$ denote a sequence of $m$-palindromic prefixes $(a_0,\dots,a_{i_k})$ for $A$ such that, for each $k\in\mathbb{N}$, there is no index $j$ such that $i_{k-1}<j<i_{k}$ and for which $(a_0,\dots,a_j)$ is an $m$-palindrome. Define the \textit{$m$-palindromic density} of $\alpha$ to be the number
\[d_m(\alpha)=\limsup_{k\to+\infty}\frac{|P_k(A)|}{|P_{k+1}(
A)|}.\]

\begin{example}\label{X:non_purely_stammering}
Here we provide an example of a transcendental $2$-palindrome with bounded partial quotients and positive $2$-palindromic density that does not satisfy Condition $(*)_w$ for any $w>1$. To this end, define the $2$-palindromic sequences $C=(1,1,0)$ and $D=(2,1)$. For each $k\in\mathbb{N}$, define $U_k=T_{C,k}(D)\in\mathbb{N}_0^*$ using the sequence $n_k=k$ as in (\ref{E:perturbed_S}). Note that, for each $k\in\mathbb{N}$, the extended continued fractions $[U_k]$ and $[\widetilde U_k]$ are well-defined and positive. Therefore, by Remark \ref{R:extended}, we may apply Lemma \ref{L:pert_symm_middle} to conclude that, for each $k\in \mathbb{N}$, the sequence $U_k$ is a $2$-palindrome. We claim that the simplification of each $U_k$ remains $2$-palindromic. To verify this claim, define $E=(1,1,3,1)\in\mathbb{N}^*$, and, for each $k\in\mathbb{N}$, define $F_k=(1,2,\dots,2,3)\in\mathbb{N}^*$. Here the middle $k$ terms of $F_k$ are all equal to $2$. For each $k\in \mathbb{N}$, we use $E$ and $F_k$ to define
\[G_k=(2,S_{F_k}\circ\dots\circ S_{F_1}(E))\in\mathbb N^*.\] 
A straightforward induction argument confirms that, for each $k\in \mathbb{N}$, the sequences $G_k$ and $\widetilde{G}_k$ are simplifications of $U_{k+1}$ and $\widetilde{U}_{k+1}$, respectively. In particular, for each $k\in\mathbb{N}$, we have $2[\widetilde{G}_k]=2[\widetilde{U}_{k+1}]=[U_{k+1}]=[G_k]$.

Let $G\in\mathbb{N}^{\mathbb{N}_0}$ denote the sequence such that, for every $k\in\mathbb{N}$, the sequence $G_k$ is a prefix for $G$. Note that $G$ is well-defined because each $G_{k+1}$ begins with $G_k$. Since, for every $k\in\mathbb{N}$, the sequence $G_k$ is a $2$-palindrome, we conclude that the sequence $G$ is also a $2$-palindrome.  

To see that $[G]$ has positive $2$-palindromic density, we first note that, for each $k\in\mathbb N$, we have $|G_k|>2^k$. Furthermore, $|G_k|=2|G_{k-1}|+O(k)$. It follows that $|G_k|/|G_{k+1}|\to1/2$ as $k\to+\infty$, and so $d_2([G])\geq1/2$.

Fix $k\in\mathbb{N}$, and let $X$ denote a particular appearance of $C$ in $T_{C,k+1}$. The \textit{maximal string of $C$'s} containing $X$ is defined to be the longest sequence of the form $C^j$ containing $X$, where $j\in\mathbb{N}$. Clearly, every maximal string of $C$'s is followed by a copy of $D$. Furthermore, it can be inductively verified that every appearance of $D$ in $T_{C,k+1}(D)$ is preceded by a copy of $C$ (except for the the first appearance of $D$). Therefore, we conclude that the only appearance of the sequence $(2,1)$ in the simplification $G_k$ occurs in its first two terms. It follows that $[G]$ does not satisfy Condition $(*)_w$ or $(*)_{\widehat w}$ for any real number $w>1$. 
\end{example}

\begin{example}
Next, we provide an example of a transcendental $2$-palindrome that fails to satisfy Condition $(*)_w$, $(*)_{\widehat w}$, or $(**)_{w,w'}$ for any real numbers $w>1$ and $w'\geq0$. In other words, we exhibit a $2$-palindrome that exhibits no large repetitive patterns. To this end, let $\{l_k\}_{k=1}^{+\infty}$ denote a strictly increasing sequence of positive integers. For each $k\geq2$, define
\begin{align*}
B_k&=\prod_{i=1}^{d_k}(2(l_{k-1}+i),l_{k-1}+i)\prod_{i=0}^{d_k-1}(2(l_k-i),l_k-i)\\
&=(b^{(k)}_0,b^{(k)}_1,\dots,b^{(k)}_{r_k})\in\mathbb N^*.
\end{align*}
Here $d_k=l_k-l_{k-1}$ and, for $X_i\in\mathbb N^*$, the product $\prod_{i=1}^n X_i$ denotes $X_1\dots X_n\in\mathbb N^*$. Note that, if $i\not=j$, the set of odd-index terms in $B_i$ is disjoint from the set of odd-index terms in $B_j$.

By Lemmas \ref{L:periodic_m_pals} and \ref{L:pert_symm_middle}, each sequence $B_k$ is a $2$-palindrome. For each $k\geq3$, define
\[T_k=S_{B_k}\circ S_{B_{k-1}}\circ \dots\circ S_{B_3}(B_2).\]
Again referring to Lemmas \ref{L:periodic_m_pals} and \ref{L:pert_symm_middle}, each $T_{k}$ is a $2$-palindrome. Since each $T_k$ is a prefix of $T_{k+1}$, the sequence $\{[T_k]\}_{k=1}^{+\infty}$ converges to a $2$-palindromic continued fraction denoted by $[T]$. Since the sequence $T$ contains arbitrarily large terms, Theorem \ref{T:main} tells us that $[T]$ is transcendental.

We now verify that $[T]$ does not satisfy Condition $(*)_w$, $(*)_{\widehat w}$, or $(**)_{w,w'}$ for any real numbers $w>1$ and $w'\geq0$, provided that the sequence $\{l_k\}_{k=1}^{+\infty}$ grows fast enough. Indeed, set $l_1=1$ and $l_2=2$. For $k\geq 3$, recursively define $l_k$ so that $|B_k|\geq2^k|T_{k-1}|$. Let $V$ denote any finite sequence contained in the sequence $T$. Let $k\in\mathbb N$ denote the smallest index such that $V$ is contained in a copy of $T_k$. We denote this copy of $T_k$ by $X$, and write $X=YUY$, where $Y$ is a copy of $T_{k-1}$ and $U$ is a copy of $B_k$. We note that $V$ intersects $U$, else the index $k$ is not minimal. If $V$ does not contain $U$, then $V$ cannot exhibit nontrivial repetition in $T$. This is because nontrivial repetition requires $V$ to begin and end in distinct copies of $B_i$, for some $i\in\mathbb N$. Therefore, we can assume that $V$ contains $U$. Since $X$ is succeeded by a copy of $B_{k+j}$ (for some $j\geq1$), and $X$ does not contain a copy of $B_{k+j}$, we conclude that any repeated prefix of $V$, denoted by $V'$, must be contained in the second copy of $Y$ in $X$. These observations imply that
\[\frac{|V'|}{|V|}\leq\frac{|Y|}{|U|}=\frac{|T_{k-1}|}{|B_k|}\leq\frac{1}{2^k}.\]
As $|V|\to+\infty$, we have $k\to+\infty$ in the above inequality. It follows that $[T]$ does not satisfy Condition $(*)_w$, $(*)_{\widehat w}$, or $(**)_{w,w'}$ for any $w>1$, $w'\geq0$. 
\end{example}

\section{Maximal $m$-Palindromic Density}\label{S:maximal}

It has been shown by Fischler in \cite{Fischler06} that the Fibonacci sequence $(a,b,a,a,b,a,b,a,a,b,a,a,b,\dots)$ possesses maximal $1$-palindromic density amongst all infinite sequences that are not eventually periodic. Moreover, the $1$-palindromic density of the Fibonacci word is equal to $1/\varphi$, where $\varphi$ is the golden ratio. In this section, we construct $m$-palindromic sequences (for $m\geq2$) whose $m$-palindromic densities are at least $1/\varphi$. We conjecture that, in analogy with the case $m=1$, these sequences achieve the maximal $m$-palindromic density amongst all $m$-palindromic infinite sequences that are not eventually periodic. 

Let $m\in\mathbb{N}$ be fixed. As in Example \ref{X:short_pal}, given any $r,s\in\mathbb{N}$ such that $r\not=s$, the words $(rm,r)$ and $(sm,s)$ are $m$-palindromes. Given $a,b\in\mathbb{N}$ and any infinite sequence $A\in\{a,b\}^{\mathbb{N}_0}$, one may form a new sequence by replacing each occurrence of $a$ with the sequence $(rm,r)$, and each occurrence of $b$ with the sequence $(sm,s)$. Denote this new sequence by $A(m,r,s)$. 

Let $F$ denote the Fibonacci sequence in $\{a,b\}^{\mathbb{N}_0}$. So $F=(a,b,a,a,b,a,b,a,a,b,a,\dots)$, and
\[F(m,r,s)=(rm,r,sm,s,rm,r,rm,r,sm,s,rm,r,sm,s,rm,r,rm,r,sm,s,rm,r,\dots).\]
We claim that $d_m([F(m,r,s)])\geq 1/\varphi$. To verify this claim, for $n\in\mathbb{N}_0$, let $S_n$ denote the $n^{th}$ prefix of the Fibonacci sequence. The first few such finite sequences are as follows: 
\[\begin{array}{ll}
S_0=(a) & S_3=(a,b,a,a,b)\\
S_1=(a,b) & S_4=(a,b,a,a,b,a,b,a)\\
S_2=(a,b,a) & S_5=(a,b,a,a,b,a,b,a,a,b,a,a,b)\\
\end{array}\]
Let $S_n^*$ denote the sequence $S_n$ truncated by the final two digits. It is well known that, for all $n\geq2$, the sequence $S_n^*$ is a $1$-palindrome. Therefore, it follows from Lemmas \ref{L:periodic_m_pals} and \ref{L:pert_symm_middle} that each sequence $S_n^*(m,r,s)$ is an $m$-palindrome. Furthermore, $|S_n^*(m,r,s)|=2|S_n^*|=2f_{n+2}-4$. Here $f_n$ denotes the $n^{th}$ Fibonacci number. Since each $S_n^*(m,r,s)$ is a prefix to $S_{n+1}^*(m,r,s)$, we conclude that 
\[d_m([F(m,r,s)])\geq\lim_{n\to+\infty}\frac{|S_n^*(m,r,s)|}{|S_{n+1}^*(m,r,s)|}=\lim_{n\to+\infty}\frac{2f_{n+2}-4}{2f_{n+3}-4}=\frac{1}{\varphi}.\]

\begin{conjecture}
For $m\in\mathbb{N}$ and $A\in\mathbb{N}^{\mathbb{N}_0}$ an $m$-palindromic sequence that is not eventually periodic, we have $d_m([A])\leq1/\varphi$. Furthermore, for $r,s\in\mathbb{N}$ with $r\not=s$, we have $d_m([F(m,r,s)])=1/\varphi$.
\end{conjecture}

\section{Quadratic Irrational Palindromes}\label{S:equivalent}

Here we prove Corollary \ref{C:m_one} via the following the result of Burger. In order to state this result, we say that two real numbers $\alpha$ and $\beta$ are \textit{equivalent} provided that there exist integers $a$, $b$, $c$, and $d$ such that 
\[\alpha=\frac{a+b\beta}{c+d\beta}\qquad ad-bc=\pm1.\]

\begin{theorem}[Burger; \cite{Burger05}]\label{T:Burger}
Let $\alpha=[A]$ denote a quadratic irrational continued fraction. The number $\alpha$ is equivalent to its algebraic conjugate if and only if $A$ is eventually periodic with period consisting of one or two $1$-palindromes. 
\end{theorem}

\begin{proof}[Proof of Corollary \ref{C:m_one}]
Assume $\alpha=[a_0,a_1,\dots]=[A]$ is an algebraic irrational $1$-palindrome. Theorem \ref{T:main} tells us via Galois' Theorem that $A$ must be purely periodic. In particular, $A=\overline{W}$ for some $W\in\mathbb N^*$. We write $W=(w_1,w_2,\dots,w_n)$ for some $n\in\mathbb N$. Let $k$ denote any index such that $(a_0,\dots,a_k)=A_k$ is $1$-palindromic. For $k$ large enough, we may write $A_k=W^jW'$ for some $j\geq1$. Here $W'$ is a prefix of $W$, and we allow for the possibility that $W'=W$. We consider separately the cases that $W'=W$ and $W'\not=W$.

We first consider the case that $W'= W$. Since $A_k$ is a $1$-palindrome, this implies that $\widetilde W=W$, and so $W$ is a $1$-palindrome. In other words, $\alpha$ is purely periodic, with $1$-palindromic period. By Theorem \ref{T:Burger}, the number $\alpha$ is equivalent to its algebraic conjugate. 

We now consider the case that $W'\not=W$. Since $A_k$ is a $1$-palindrome, $\widetilde W'$ is a prefix of $W$, and there exists $r<n$ such that $\widetilde W'=(w_1,w_2,\dots,w_r)$. Since $W'$ is also a prefix of $W$, we have $W'=(w_1,w_2,\dots,w_r)$. In other words, $W'$ is a $1$-palindrome. Write $W''=(w_{r+1},\dots,w_n)$. Again using the assumption that $A_k$ is a $1$-palindrome, we find that $(w_n,\dots,w_{r+1})=(w_{r+1},\dots,w_n)$. In other words, $W''$ is also a $1$-palindrome. Therefore, $A$ has period $W=W'W''$, which consists of two $1$-palindromes. By Theorem \ref{T:Burger}, the number $\alpha$ is equivalent to its algebraic conjugate.
\end{proof}

\begin{example}\label{X:non_equivalent}
Let $\alpha=[A]=[2,\overline{D}]=[2,\overline{1,1,2,2,3}]$. To verify that $\alpha$ is a $2$-palindrome, we proceed as follows. Write $B=(2,1)$ and $C=(1,2,2,1,0)$. For each $k\in\mathbb{N}$, define $T_k=S_C^k(B)$. Note that $[C]$ is a well-defined and positive extended continued fraction, and that the same is true of each $[T_k]$. By Remark \ref{R:extended} and Lemma \ref{L:pert_symm_middle}, each sequence $T_k$ is a $2$-palindrome. Furthermore, one can verify by way of induction that, for each $k\in\mathbb{N}$, a simplification of $T_k$ is given by $T_k'=(2,D^{n_k},1)$, where $n_1=1$ and, for each integer $i\geq 2$, we recursively define $n_i=2n_{i-1}+1$. Therefore, for every $k\in\mathbb{N}$, the sequence $T'_k$ is a  prefix for $A$. Finally, one can verify that $\widetilde{T}_k'=(1,\widetilde{D}^{n_k},2)$ is a simplification for $\widetilde{T}_k$, and so $2[\widetilde{T}_k']=[T_k']$. In other words, $A$ is a $2$-palindrome.
\end{example}

\begin{remark}
We first note that Example \ref{X:non_equivalent} illustrates Theorem \ref{T:main} in the sense that $\alpha/2$ has continued fraction expansion $[\overline{1,3,2,2,1}]$, and is therefore reduced, while $\alpha$ itself is not reduced. We then note that neither $\alpha$ nor $\alpha/2$ can be written as a continued fraction with period consisting of two $1$-palindromes. By Theorem 8.1, neither $\alpha$ nor $\alpha/2$ is equivalent to its algebraic conjugate. As pointed out in the introduction, this indicates that Corollary \ref{C:m_one} does not have a straightforward generalization to $m$-palindromes when $m\geq2$. 
\end{remark}

\bibliographystyle{amsalpha} 
\bibliography{bib}  

\end{document}